\documentclass[a4paper,10pt]{article}

\usepackage[utf8]{inputenc} 
\usepackage{amsmath}
\usepackage{amsbsy}
\usepackage{amssymb}
\usepackage{amscd,amsthm}
\usepackage{graphicx}
\usepackage[mathcal]{eucal}
\usepackage{verbatim}
\usepackage[english]{babel}
\usepackage[dvigs]{epsfig}
\usepackage{dsfont}
\usepackage{marvosym}
\usepackage{anysize}

\newcommand{\la}{\lambda}

\newcommand{\Z}{\mathds{Z}}
\newcommand{\Q}{\mathds{Q}}
\newcommand{\N}{\mathds{N}}
\newcommand{\R}{\mathds{R}}

\newcommand{\p}{\phantom}
\newcommand{\q}{\quad}

\newcommand{\nn}{\nonumber}

\newtheorem{thm}{Theorem}[section]
\newtheorem{lem}[thm]{Lemma}
\newtheorem{kor}[thm]{Corollary}

\newtheorem{prop}[thm]{Proposition}

\theoremstyle{definition}
\newtheorem*{defi}{Definition}

\theoremstyle{remark}

\title{On a sequence involving the prime numbers}
\author{Christian Axler}

\begin{document}

\maketitle

\begin{abstract}
In this paper we study a sequence involving the prime numbers by deriving two asymptotic formulas and finding new upper and lower bounds, which improve the
currently known estimates.
\end{abstract}

\section{Introduction}

In this paper, we study the difference
\begin{displaymath}
C_n = np_n - \sum_{k \leq n} p_k
\end{displaymath}
(see also \cite{pol}), where $p_n$ is the $n$th prime number, by proving two asymptotic formulas and finding lower and upper bounds for $C_n$.

\section{Two asymptotic formulas for $C_n$}

Let $m \in \N$. By \cite{cp}, there exist unique $a_{is} \in \Q$, where $a_{ss} = 1$ for all $1 \leq s \leq m$, such that
\begin{equation} \label{201}
p_n = n\left( \log n + \log \log n - 1 + \sum_{s=1}^m \frac{(-1)^{s+1}}{s\log^s n} \sum_{i=0}^s a_{is}(\log \log n)^i \right) + O (c_m(n)),
\end{equation}
where
\begin{displaymath}
c_m(n) = \frac{n(\log \log n)^{m+1}}{\log^{m+1} n}.
\end{displaymath}
We set
\begin{displaymath}
h_m(n) = \sum_{j=1}^{m} \frac{(j-1)!}{2^j \log^j n}.
\end{displaymath}
Further, we recall the following definition from \cite{ax2}.

\begin{defi}
Let $s,i,j,r \in \N_0$ with $j \geq r$. We define the integers $b_{s,i,j,r} \in \Z$ as follows:
\begin{itemize}
 \item If $j=r=0$, then
\begin{equation} \label{202}
b_{s,i,0,0} = 1.
\end{equation}
 \item If $j \geq 1$, then
\begin{equation} \label{203}
b_{s,i,j,j} = b_{s,i,j-1,j-1} \cdot (-i+j-1).
\end{equation}
 \item If $j \geq 1$, then
\begin{equation} \label{204}
b_{s,i,j,0} = b_{s,i,j-1,0} \cdot (s+j-1).
\end{equation}
 \item If $j > r \geq 1$, then
\begin{equation} \label{205}
b_{s,i,j,r} = b_{s,i,j-1,r} \cdot (s+j-1) + b_{s,i,j-1,r-1} \cdot (-i+r-1).
\end{equation}
\end{itemize}
\end{defi}

\noindent
Using \eqref{201} and Theorem 2.5 of \cite{ax2}, we obtain the first asymptotic formula for $C_n$.

\begin{thm} \label{t201}
Let $m \in \N$. Then,
\begin{align*}
C_n & = \frac{n^2}{2} \left( \log n + \log \log n - \frac{1}{2} + h_m(n) \right)\\
& \p{\q\q} + \frac{n^2}{2} \sum_{s=1}^{m} \frac{(-1)^{s+1}}{s\log^sn} \sum_{i=0}^s a_{is} \left( 2(\log \log n)^i - \sum_{j=0}^{m-s}
\sum_{r=0}^{\min\{i,j\}} \frac{b_{s,i,j,r}(\log \log n)^{i-r}}{2^j\log^jn} \right) + O (nc_m(n)).
\end{align*}
\end{thm}

\begin{proof}
First, we multiply the asymptotic formula \eqref{201} with $n$. Then, we subtract the asymptotic formula for $\sum_{k \leq n}p_k$ from \cite[Theorem
2.5]{ax2} to obtain our proposition.
\end{proof}

\begin{kor} \label{kor407}
Let $m \in \N$. Then there are unique monic polynomials $U_s \in \Q[x]$, where $1 \leq s \leq m$ and $\emph{deg}(U_s) = s$, such that
\begin{displaymath}
C_n = \frac{n^2}{2} \left( \log n + \log \log n - \frac{1}{2} + \sum_{s=1}^{m} \frac{(-1)^{s+1}U_s(\log \log n)}{s\log^sn} \right) + O (nc_m(n)).
\end{displaymath}
In particular, we have $U_1(x) = x - 3/2$ and $U_2(x) = x^2-5x+15/2$.
\end{kor}

\begin{proof}
Since $a_{ss} = 1$ and $b_{s,s,0,0} = 1$, the first claim follows from Theorem \ref{t201}. Now let $m=2$. By \cite{cp}, we have $a_{01}=-2$, $a_{11}=1$,
$a_{02}=11$, $a_{12} = -6$ and $a_{22}=1$. Further, we use the formulas \eqref{202}--\eqref{205} to compute the integers $b_{s,i,j,r}$. Then, using Theorem
\ref{t201}, we obtain the polynomials $U_1$ and $U_2$.
\end{proof}

\noindent
To find another asymptotic formula for $C_n$, we obtain the following identity, which leads to a possibility to estimate $C_n$ by using estimates for
$\pi(x)$.

\begin{lem} \label{l203}
For all $n \in \N$,
\begin{displaymath}
C_n = \int_{2}^{p_{n}}{\pi(x) \, dx}.
\end{displaymath}
\end{lem}

\begin{proof}
See \cite{pd}.
\end{proof}

\noindent
Now we give certain rules of integration.

\begin{lem} \label{l204}
Let $x,a \in \R$ with $x \geq a > 1$. Then,
\begin{displaymath}
\int_a^x \frac{t\, dt}{\log t} = \emph{li}(x^2) - \emph{li}(a^2).
\end{displaymath}
\end{lem}

\begin{proof}
See \cite[Lemme 1.6]{pd}.
\end{proof}

\begin{lem} \label{l205}
Let $x,a \in \R$ with $x \geq a > 1$. Then,
\begin{displaymath}
\int_a^x \frac{t \, dt}{\log^2 t} = 2\, \emph{li}(x^2)-2\, \emph{li}(a^2)- \frac{x^2}{\log x}+ \frac{a^2}{\log a}.
\end{displaymath}
\end{lem}

\begin{proof}
See \cite[Lemme 1.6]{pd}.
\end{proof}

\begin{lem} \label{l206}
Let $r,s \in \mathds{R}$ with $s \geq r > 1$ and $n \in \mathds{N}$. Then,
\begin{displaymath}
\int_r^s \frac{x \, dx}{\log^{n+1} x} = \frac{r^2}{n \log ^n r} - \frac{s^2}{n \log ^n s} + \frac{2}{n} \int_r^s \frac{x}{\log^{n} x} \; dx.
\end{displaymath}
\end{lem}

\begin{proof}
Integration by parts.
\end{proof}

\begin{lem} \label{l207}
Let $r,s\in \R$ with $s \geq r > 1$. Then, for all $m \in \mathds{N}$ with $m \geq 2$ we have
\begin{displaymath}
\int_r^s \frac{x \, dx}{\log^m x} = \frac{2^{m-2}}{(m-1)!} \int_r^s \frac{x \, dx}{\log^2 x} - \sum_{k=2}^{m-1} \frac{2^{m-1-k}(k-1)!}{(m-1)!} \left(
\frac{s^2}{\log^k s} - \frac{r^2}{\log^k r} \right).
\end{displaymath}
\end{lem}

\begin{proof}
By induction on $m$.
\end{proof}

\noindent
The next proposition plays an important role for the proof of the second asymptotic formula for $C_n$.

\begin{prop} \label{p208}
Let $m\in\N$ with $m \geq 2$. Let $a_2, \ldots, a_m \in \R$ and $r,s\in\R$ with $s \geq r > 1$. Then,
\begin{displaymath}
\sum_{k=2}^m a_k \int_r^s \frac{x \, dx}{\log^k x} = t_{m-1,1} \int_r^s \frac{x \, dx}{\log^2 x} - \sum_{k=2}^{m-1} t_{m-1,k} \left( \frac{s^2}{\log^k s} -
\frac{r^2}{\log^k r} \right),
\end{displaymath}
where
\begin{equation} \label{gl419}
t_{i,j} := (j-1)! \sum_{l=j}^{i} \frac{2^{l-j}a_{l+1}}{l!}.
\end{equation}
\end{prop}

\begin{proof}
If $m = 2$, the claim is obviously true. By induction hypothesis, we have
\begin{displaymath}
\sum_{k=2}^{m+1} a_k \int_r^s \frac{x \, dx}{\log^k x} = t_{m-1,1} \int_r^s \frac{x \, dx}{\log^2 x} - \sum_{k=2}^{m-1} t_{m-1,k} \left( \frac{s^2}{\log^k s} - \frac{r^2}{\log^k r} \right) + a_{m+1} \int_r^s \frac{x \, dx}{\log^{m+1} x}.
\end{displaymath}
By Lemma \ref{l206}, we get
\begin{align*}
\sum_{k=2}^{m+1} a_k \int_r^s \frac{x \, dx}{\log^k x}&  = t_{m-1,1} \int_r^s \frac{x \, dx}{\log^2 x} - \sum_{k=2}^{m-1} t_{m-1,k} \left( \frac{s^2}{\log^k s} - \frac{r^2}{\log^k r} \right) + \frac{2a_{m+1}}{m} \int_r^s \frac{x \, dx}{\log^m x} \\
& \p{\q\q} - \frac{a_{m+1}s^2}{m \log^m s} + \frac{a_{m+1}r^2}{m \log^m r}.
\end{align*}
Now we can use Lemma \ref{l207} and the equality $t_{m-1,1} + 2^{m-1}a_{m+1}/m! = t_{m,1}$ to obtain
\begin{align*}
\sum_{k=2}^{m+1} a_k \int_r^s \frac{x \, dx}{\log^k x} & = t_{m,1} \int_r^s \frac{x \, dx}{\log^2 x} - \sum_{k=2}^{m-1} \left(
\frac{2^{m-k}a_{m+1}(k-1)!}{m!} + t_{m-1,k} \right) \left( \frac{s^2}{\log^k s} - \frac{r^2}{\log^k r} \right) \\
& \p{\q\q} - \frac{a_{m+1}(m-1)!}{m!} \left( \frac{s^2}{\log^m s} - \frac{r^2}{\log^m r} \right).
\end{align*}
Since we have
\begin{displaymath}
\frac{2^{m-k}a_{m+1}(k-1)!}{m!} + t_{m-1,k} = t_{m,k}
\end{displaymath}
and $t_{m,m} = a_{m+1}(m-1)!/(m!)$, our proposition is proved.
\end{proof}

\noindent
Now we give another asymptotic formula for $C_n$.

\begin{thm} \label{t209}
Let $m \in \N$. Then,
\begin{equation} \label{gl421}
C_n = \sum_{k=1}^{m-1} (k-1)! \left(1 - \frac{1}{2^k} \right) \frac{p_n^2}{\log^k p_n} + O \left( \frac{p_n^2}{\log^m p_n} \right).
\end{equation}
\end{thm}

\begin{proof}
First we recall a well-known asymptotic formula for the prime counting function $\pi(x)$; i.e.
\begin{equation} \label{208}
\pi(x) = \frac{x}{\log x} + \frac{x}{\log^2 x} + \frac{2x}{\log^3 x} + \frac{6x}{\log^4 x} + \ldots + \frac{(m-1)!x}{\log^m x} + O \left(
\frac{x}{\log^{m+1} x} \right).
\end{equation}
Using \eqref{208} and Lemma \ref{l203}, we get
\begin{displaymath}
C_n = \sum_{k=1}^m (k-1)! \int_2^{p_n} \frac{x \, dx}{\log^k x} + O \left( \int_2^{p_n} \frac{x \, dx}{\log^{m+1} x} \right).
\end{displaymath}
Integration by parts gives
\begin{displaymath}
C_n = \sum_{k=1}^m (k-1)! \int_2^{p_n} \frac{x \, dx}{\log^k x} + O \left( \frac{p_n^2}{\log^m p_n} \right).
\end{displaymath}
We can apply Proposition \ref{p208} to get
\begin{displaymath}
C_n = \int_2^{p_n} \frac{x \, dx}{\log x} + (2^{m-1}-1) \int_2^{p_n} \frac{x \, dx}{\log^2 x}- \sum_{k=2}^{m-1} \left( \frac{(k-1)! (2^{m-k} -
1)p_n^2}{\log^k p_n} \right) + O \left( \frac{p_n^2}{\log^m p_n} \right).
\end{displaymath}
Using Lemma \ref{l204} and Lemma \ref{l205}, we get
\begin{displaymath}
C_n = (2^m - 1) \, \text{li}(p_n^2)- \sum_{k=1}^{m-1} \left( \frac{(k-1)! (2^{m-k} - 1)p_n^2}{\log^k p_n} \right) + O \left( \frac{p_n^2}{\log^m p_n}
\right).
\end{displaymath}
Now we use the asymptotic formula
\begin{equation} \label{209}
\text{li}(x) = \frac{x}{\log x} + \frac{x}{\log^2 x} + \frac{2x}{\log^3 x} + \frac{6x}{\log^4 x} + \ldots + \frac{(m-1)!x}{\log^m x} + O \left(
\frac{x}{\log^{m+1} x} \right),
\end{equation}
which can be showed by integration by parts, to obtain the equality
\begin{displaymath}
C_n = (2^m - 1)\sum_{k=1}^{m-1} \frac{(k-1)!\, p_n^2}{2^k \log^k p_n} - \sum_{k=1}^{m-1} \left( \frac{(k-1)! (2^{m-k} - 1)p_n^2}{\log^k p_n} \right) + O
\left( \frac{p_n^2}{\log^m p_n} \right).
\end{displaymath}
and our theorem is proved.
\end{proof}

\noindent
Using \eqref{208}, we get the following corollary.

\begin{kor} \label{k210}
Let $m \in \N$. Then,
\begin{displaymath}
\sum_{k \leq n} p_k = \pi(p_n^2) + O \left( \frac{p_n^2}{\log^m p_n} \right).
\end{displaymath}
\end{kor}

\begin{proof}
From Theorem \ref{t209} and the definition of $C_n$ it follows that
\begin{displaymath}
\sum_{k \leq n} p_k = np_n - \sum_{k=1}^{m-1} \frac{(k-1)!\,p_n^2}{\log^kp_n} + \sum_{k=1}^{m-1}\frac{(k-1)!\,p_n^2}{2^k\log^k p_n} + O \left(
\frac{p_n^2}{\log^m p_n} \right).
\end{displaymath}
Since $n=\pi(p_n)$, we obtain
\begin{displaymath}
\sum_{k \leq n} p_k = \pi(p_n)p_n - \sum_{k=1}^{m-1} \frac{(k-1)!\,p_n^2}{\log^kp_n} + \sum_{k=1}^{m-1}\frac{(k-1)!\,p_n^2}{2^k\log^k p_n} + O \left(
\frac{p_n^2}{\log^m p_n} \right).
\end{displaymath}
Using \eqref{208}, we get the equality
\begin{displaymath}
\sum_{k \leq n} p_k = \sum_{k=1}^{m-1}\frac{(k-1)!\,p_n^2}{2^k\log^k p_n} + O \left( \frac{p_n^2}{\log^m p_n} \right) = \pi(p_n^2) + O \left(
\frac{p_n^2}{\log^m p_n} \right)
\end{displaymath}
and the corollary is proved.
\end{proof}

\noindent
Comparing \eqref{208} and \eqref{209}, we see that $\pi(x)$ and $\text{li}(x)$ have the same asymptotic formaula. Hence, using Corollary \ref{k210}, we
also get the following result on the sum of the first $n$ prime numbers.

\begin{kor} \label{k211}
Let $m \in \N$. Then,
\begin{displaymath}
\sum_{k \leq n} p_k = \emph{li}(p_n^2) + O \left( \frac{p_n^2}{\log^m p_n} \right).
\end{displaymath}
\end{kor}

\section{A lower bound for $C_n$}

Let $m \in \N$ with $m \geq 2$ and let $a_2, \ldots, a_m$, $x_0$, $y_0 \in \R$, so that
\begin{equation} \label{310}
\pi(x) \geq \frac{x}{\log x} + \sum_{k=2}^m \frac{a_kx}{\log^k x}
\end{equation}
for every $x \geq x_0$ and
\begin{equation} \label{311}
\text{li}(x) \geq \sum_{j=1}^{m-1} \frac{(j-1)! x}{\log^j x}
\end{equation}
for every $x \geq y_0$. Then, we obtain the following lower bound for $C_n$.

\begin{thm} \label{t301}
If $n \geq \max \{ \pi(x_0) + 1, \pi(\sqrt{y_0}) + 1 \}$, then
\begin{displaymath}
C_n \geq d_0 + \sum_{k=1}^{m-1} \left( \frac{(k-1)!}{2^k} ( 1 + 2t_{k-1,1} ) \right) \frac{p_n^2}{\log^k p_n},
\end{displaymath}
where $t_{i,j}$ is defined as in \eqref{gl419} and $d_0$ is given by
\begin{displaymath}
d_0 = d_0(m,a_2,\ldots, a_m, x_0) = \int_2^{x_0} \pi(x) \, dx - ( 1 + 2 t_{m-1,1} )\, \emph{li}(x_0^2) + \sum_{k=1}^{m-1} t_{m-1,k} \frac{x_0^2}{\log^k
x_0}.
\end{displaymath}
\end{thm}

\begin{proof}
Since $p_n \geq x_0$, we use Lemma \ref{l203} and \eqref{310} to obtain
\begin{displaymath}
C_n \geq \int_2^{x_0} \pi(x) \, dx + \int_{x_0}^{p_n} \frac{x \, dx}{\log x} + \sum_{k=2}^m a_k \int_{x_0}^{p_n} \frac{x \, dx}{\log^k x}.
\end{displaymath}
Now, we apply Lemma \ref{l204} and Proposition \ref{p208} to get
\begin{displaymath}
C_n \geq \int_2^{x_0} \pi(x) \, dx - \text{li}(x_0^2) + \text{li}(p_n^2) + t_{m-1,1} \int_{x_0}^{p_n} \frac{x \, dx}{\log^2 x} - \sum_{k=2}^{m-1} t_{m-1,k}
\left( \frac{p_n^2}{\log^k p_n} - \frac{x_0^2}{\log^k x_0} \right).
\end{displaymath}
Using Lemma \ref{l205}, we obtain
\begin{displaymath}
C_n \geq d_0 + \left( 1 + 2 t_{m-1,1} \right)\, \text{li}(p_n^2) - \sum_{k=1}^{m-1} t_{m-1,k} \frac{p_n^2}{\log^k p_n}.
\end{displaymath}
Since $p_n^2 \geq y_0$, we use \eqref{311} to conclude
\begin{displaymath}
C_n \geq d_0 + \sum_{k=1}^{m-1} \left( \frac{(k-1)!}{2^k} + \frac{(k-1)!}{2^{k-1}}\,t_{m-1,1} - t_{m-1,k} \right)\frac{p_n^2}{\log^k p_n}
\end{displaymath}
and it remains to use the definition of $t_{ij}$.
\end{proof}

\section{An upper bound for $C_n$}

Next, we derive for the first time an upper bound for $C_n$. Let $m \in \N $ with $m \geq 2$ and let $a_2, \ldots, a_m, x_1 \in \R$ so that
\begin{equation} \label{412}
\pi(x) \leq \frac{x}{\log x} + \sum_{k=2}^m \frac{a_kx }{\log^k x}
\end{equation}
for every $x \geq x_1$ and let $\la, y_1 \in \R$ so that
\begin{equation} \label{413}
\text{li}(x) \leq \sum_{j=1}^{m-2} \frac{(j-1)! x}{\log^j x} + \frac{\la x}{\log^{m-1} x}
\end{equation}
for every $x \geq y_1$. Setting
\begin{displaymath}
d_1 := d_1(m,a_2, \ldots, a_m, x_1) =  \int_2^{x_1} \pi(x) \, dx - ( 1 + 2 t_{m-1,1} ) \, \text{li}(x_1^2) + \sum_{k=1}^{m-1} t_{m-1,k} \frac{x_1^2}{\log^k x_1},
\end{displaymath}
where $t_{m-1,k}$ is defined by \eqref{gl419}, we obtain the following

\begin{thm} \label{t401}
If $n \geq \max \{ \pi(x_1) + 1, \pi(\sqrt{y_1}) + 1 \}$, then
\begin{displaymath}
C_n \leq d_1 + \sum_{k=1}^{m-2} \left( \frac{(k-1)!}{2^k} ( 1 + 2t_{k-1,1} ) \right) \frac{p_n^2}{\log^k p_n} + \left( \frac{(1 + 2t_{m-1,1})\la}{2^{m-1}} - \frac{a_m}{m-1} \right) \frac{p_n^2}{\log^{m-1} p_n}.
\end{displaymath}
\end{thm}

\begin{proof}
Since $p_n \geq x_1$, we use Lemma \ref{l203} and \eqref{412} to get
\begin{displaymath}
C_n \leq \int_2^{x_1} \pi(x) \, dx + \int_{x_1}^{p_n} \frac{x \, dx}{\log x} + \sum_{k=2}^m a_k \int_{x_1}^{p_n} \frac{x \, dx}{\log^k x}.
\end{displaymath}
We apply Lemma \ref{l204} and Proposition \ref{p208} to obtain
\begin{displaymath}
C_n \leq \int_2^{x_1} \pi(x) \, dx - \text{li}(x_1^2)+ \text{li}(p_n^2) + t_{m-1,1} \int_{x_1}^{p_n} \frac{x \, dx}{\log^2 x} - \sum_{k=2}^{m-1} t_{m-1,k} \left( \frac{p_n^2}{\log^k p_n} - \frac{x_1^2}{\log^k x_1} \right).
\end{displaymath}
Using Lemma \ref{l205}, we get
\begin{displaymath}
C_n \leq d_1 + ( 1 + 2 t_{m-1,1})\, \text{li}(p_n^2) - \sum_{k=1}^{m-1} t_{m-1,k} \frac{p_n^2}{\log^k p_n}.
\end{displaymath}
Now we can use the inequality \eqref{413} to obtain
\begin{displaymath}
C_n \leq d_1 + \sum_{k=1}^{m-2} \left( \frac{(k-1)!}{2^k} + \frac{t_{m-1,1}(k-1)!}{2^{k-1}} - t_{m-1,k} \right) \frac{p_n^2}{\log^k p_n} + \left( \frac{(1 + 2t_{m-1,1})\la}{2^{m-1}} - t_{m-1,m-1} \right) \frac{p_n^2}{\log^{m-1} p_n}
\end{displaymath}
and it remains to use the definition of $t_{ij}$.
\end{proof}

\section{Numerical results}

By setting $m=8$ in Theorem \ref{t209}, we obtain
\begin{displaymath}
C_n = \frac{p_n^2}{2 \log p_n} + \frac{3p_n^2}{4\log^2 p_n} + \frac{7p_n^2}{4 \log^3 p_n} + \chi(n) + O \left( \frac{p_n^2}{\log^8 p_n} \right),
\end{displaymath}
where $\chi(n)$ is defined by
\begin{displaymath}
\chi(n) = \frac{45 p_n^2}{8\log^4 p_n} + \frac{93p_n^2}{4\log^5 p_n} + \frac{945p_n^2}{8 \log^6 p_n} + \frac{5715p_n^2}{8 \log^7 p_n}.
\end{displaymath}

\subsection{An explicit lower bound for $C_n$}

Dusart \cite{pd} proved, that
\begin{equation} \label{510}
C_n \geq c + \frac{p_n^2}{2\log p_n} + \frac{3p_n^2}{4\log^2 p_n}
\end{equation}
for every $n \geq 109$, where $c \approx -47.1$. The goal of this subsection is to improve inequality \eqref{510}. In order to do this, we first give two
lemmata concerning explicit estimates for $\text{li}(x)$.

\begin{lem} \label{l501}
If $x \geq 4171$, then
\begin{displaymath}
\emph{li}(x) \geq \frac{x}{\log x} + \frac{x}{\log^2 x} + \frac{2x}{\log^3 x} + \frac{6x}{\log^4 x} + \frac{24x}{\log^5 x} + \frac{120x}{\log^6 x} +
\frac{720x}{\log^7 x} + \frac{5040x}{\log^8 x}.
\end{displaymath}
\end{lem}

\begin{proof}
We denote the right hand side by $\alpha(x)$. Let $f(x) = \text{li}(x) - \alpha(x)$. Then, $f(4171) \geq 0.00019$ and $f'(x) = 40320/\log^9 x$, and our
lemma is proved.
\end{proof}

\begin{lem} \label{l502}
If $x \geq 10^{16}$, then
\begin{displaymath}
\emph{li}(x) \leq \frac{x}{\log x} + \frac{x}{\log^2 x} + \frac{2x}{\log^3 x} + \frac{6x}{\log^4 x} + \frac{24x}{\log^5 x} + \frac{120x}{\log^6 x} +
\frac{900x}{\log^7 x}.
\end{displaymath}
\end{lem}

\begin{proof}
Similarly to the proof of Lemma \ref{l501}.
\end{proof}

\noindent
Setting
\begin{equation} \label{gl429}
\Theta(n) = \frac{43.6p_n^2}{8\log^4 p_n} + \frac{90.9p_n^2}{4\log^5 p_n} + \frac{927.5p_n^2}{8\log^6 p_n} + \frac{702.5625p_n^2}{\log^7 p_n} +
\frac{4942.21875p_n^2}{\log^8 p_n},
\end{equation}
we get the following improvement of \eqref{510}.

\begin{prop} \label{p503}
If $n \geq 52703656$, then
\begin{displaymath}
C_n \geq \frac{p_n^2}{2 \log p_n} + \frac{3p_n^2}{4 \log^2 p_n} + \frac{7p_n^2}{4 \log^3 p_n} + \Theta(n).
\end{displaymath}
\end{prop}

\begin{proof}
We choose $m=9$, $a_2=1$, $a_3 = 2$, $a_4=5.65$, $a_5=23.65$, $a_6 = 118.25$, $a_7 = 709.5$, $a_8 = 4966.5$, $a_9 = 0$, $x_0 = 1332450001$ and $y_0 = 4171$.
By \cite{ax}, we obtain the inequality \eqref{310} for every $x \geq x_0$ and \eqref{311} hols for every $x \geq y_0$ by Lemma \ref{l501}. Substituting
these values in Theorem \ref{t301}, we get
\begin{displaymath}
C_n \geq d_0 + \frac{p_n^2}{2 \log p_n} + \frac{3p_n^2}{4 \log^2 p_n} + \frac{7p_n^2}{4 \log^3 p_n} + \Theta(n)
\end{displaymath}
for every $n \geq 66773605$, where $d_0 = d_0(9,1,2,5.65,23.65,118.25,709.5,4966.5,0, x_0)$ is given by
\begin{align*}
d_0 & = \int_2^{x_0} \pi(x) \, dx - \frac{753.1}{3} \; \text{li}(x_0^2) + \frac{375.05 x_0^2}{3 \log x_0} + \frac{186.025 x_0^2}{3 \log^2 x_0} +
\frac{183.025 x_0^2}{3 \log^3 x_0} + \frac{88.6875x_0^2}{\log^4 x_0} \\
& \p{\q\q} + \frac{165.55x_0^2}{\log^5 x_0} + \frac{354.75x_0^2}{\log^6 x_0} + \frac{709.5x_0^2}{\log^7 x_0}.
\end{align*}
Since $x_0^2 \geq 10^{16}$, we obtain using Lemma \ref{l502},
\begin{displaymath}
d_0 \geq \int_2^{x_0} \pi(x) \, dx - \frac{x_0^2}{2 \log x_0} - \frac{3 x_0^2}{4 \log^2 x_0} - \frac{7 x_0^2}{4 \log^3 x_0} - \frac{5.45x_0^2}{\log^4 x_0}
- \frac{22.725 x_0^2}{ \log^5 x_0} - \frac{115.9375x_0^2}{\log^6 x_0} - \frac{1055.578125x_0^2}{\log^7 x_0}.
\end{displaymath}
Using $\log x_0 \geq 21.01027$, we get
\begin{align} \label{gl428}
d_0 & \geq \int_2^{x_0} \pi(x) \, dx - 4.22512933 \cdot 10^{16} - 0.30164729 \cdot 10^{16} - 0.03349997 \cdot 10^{16} - 0.0049656 \cdot 10^{16} \nn \\
& \p{\q\q} - 0.00098548 \cdot 10^{16} - 0.0002393 \cdot 10^{16} - 0.0001037 \cdot 10^{16} \nn \\
& = \int_2^{x_0} \pi(x) \, dx - 4.56657067 \cdot 10^{16}.
\end{align}
Since $x_0 = p_{66773604}$, we obtain using Lemma \ref{l203} and a computer,
\begin{displaymath}
\int_2^{x_0} \pi(x) \, dx = C_{66773604} = 45665745738169817.
\end{displaymath}
Hence, by \eqref{gl428}, we get $d_0 \geq 3.9 \cdot 10^{10} > 0$. So we obtain the asserted inequality for every $n \geq 66773605$. For every $52703656 \leq
n \leq 66773604$ we check the inequality with a computer.
\end{proof}

\subsection{An explicit upper bound for $C_n$}

We begin with the following lemma.

\begin{lem} \label{lem426}
If $x \geq 10^{18}$, then
\begin{displaymath}
\emph{li}(x) \leq \frac{x}{\log x} + \frac{x}{\log^2 x} + \frac{2x}{\log^3 x} + \frac{6x}{\log^4 x} + \frac{24x}{\log^5 x} + \frac{120x}{\log^6 x} +
\frac{720x}{\log^7 x}  + \frac{6300x}{\log^8 x}.
\end{displaymath}
\end{lem}

\begin{proof}
Similarly to the proof of Lemma \ref{l501}.
\end{proof}

\noindent
Using an upper bound for $\pi(x)$ from \cite{ax}, we obtain the following explicit upper bound for $C_n$, where
\begin{equation} \label{gl433}
\Omega(n) = \frac{46.4p_n^2}{8\log^4 p_n} +  \frac{95.1p_n^2}{4\log^5 p_n} + \frac{962.5p_n^2}{8\log^6 p_n} + \frac{5809.5p_n^2}{8\log^7 p_n} +
\frac{59424p_n^2}{8\log^8 p_n}.
\end{equation}

\begin{prop} \label{kor425}
For every $n \in \N$,
\begin{displaymath}
C_n \leq \frac{p_n^2}{2 \log p_n} + \frac{3p_n^2}{4 \log^2 p_n} + \frac{7p_n^2}{4 \log^3 p_n} + \Omega(n).
\end{displaymath}
\end{prop}

\begin{proof}
We choose $a_2=1$, $a_3=2$, $a_4=6.35$, $a_5=24.35$, $a_6=121.75$, $a_7=730.5$, $a_8=6801.4$, $\lambda=6300$, $x_1 = 11$ and $y_1 = 10^{18}$. By \cite{ax},
we get that the inequality \eqref{412} holds for every $x \geq x_1$ and by Lemma \ref{lem426}, that \eqref{413} holds for all $y \geq y_1$. By substituting
these values in Theorem \ref{t401}, we get
\begin{equation} \label{gl434}
C_n \leq d_1 + \frac{p_n^2}{2 \log p_n} + \frac{3p_n^2}{4 \log^2 p_n} + \frac{7p_n^2}{4 \log^3 p_n} +  \Omega(n) - \frac{0.4375p_n^2}{8 \log^8 p_n}
\end{equation}
for every $n \geq 50847535$, where $d_1 = d_1(9,1,2,6.35,24.35,121.75,730.5,6801.4,0, x_1)$ is given by
\begin{align*}
d_1 & = \int_2^{x_1} \pi(x) \, dx - \frac{950777}{3150} \; \text{li}(x_0^2) + \frac{947627 x_0^2}{6300 \log x_0} + \frac{941327 x_0^2}{12600 \log^2 x_0} +
\frac{928727 x_0^2}{12600 \log^3 x_0} + \frac{902057 x_0^2}{8400 \log^4 x_0} \\
& \p{\q\q} + \frac{425461 x_0^2}{2100 \log^5 x_0} + \frac{187163 x_0^2}{420 \log^6 x_0} + \frac{34007 x_0^2}{35\log^7 x_0}.
\end{align*}
Since $\text{li}(x_1^2) \geq 34.59$ and $\log x_1 \geq 2.39$, we obtain $d_1 \leq 450$. We define $f(x) = 0.4375x^2/(8\log^8 x)-450$. Since $f(6\cdot 10^6)
\geq 109$ and $f'(x) \geq 0$ for every $x \geq e^4$, we get $f(p_n) \geq 0$ for every $n \geq \pi(6 \cdot 10^6) + 1 = 412850$. Now we can use \eqref{gl434}
to obtain the claim for every $n \geq 50847535$. For every $ 1 \leq n \leq 50847534$ we check the asserted inequality with a computer.
\end{proof}


\vspace{5mm}

\textsc{Mathematisches Institut, Heinrich-Heine-Universität Düsseldorf, 40225 Düsseldorf}, \textsc{Germany}

\emph{E-mail address}: \texttt{axler@math.uni-duesseldorf.de}

\end{document}